\documentclass{amsart}
\usepackage{graphicx}\usepackage{mathrsfs}
\theoremstyle{definition}

\theoremstyle{remark}

\newtheorem{theorem}{\bf Theorem}[section]

\newtheorem{lemma}[theorem]{\bf Lemma}

\newtheorem{remark}[theorem]{\bf Remark}

\numberwithin{equation}{section}

\begin{document}

\title[Finite $2$-groups of class $2$]{Finite $2$-groups of class $2$  with specific automorphism group }
\author[A. Abdollahi,  M. Ahmadi, S. M. Ghoraishi]{A. Abdollahi, M. Ahmadi, S. M. Ghoraishi }
\address[A. Abdollahi]{Department of Mathematics\\ University of Isfahan
\\ $81746-73441$ Isfahan\\ Iran\\ and \\ School of Mathematics
Institute for Research in Fundamental Sciences (IPM)\\
P.O.Box: 19395--5746 Tehran\\ Iran}
\email[First author]{a.abdollahi@math.ui.ac.ir}
\address[M. Ahmadi]{Department of Mathematics\\ University of Isfahan
\\ $81746-73441$ Isfahan\\ Iran}
\email[Second author]{m\_ahmadi@sci.ui.ac.ir}
\address[S. M. Ghoraishi]{Department of Mathematics\\ University of Isfahan
\\ $81746-73441$ Isfahan\\ Iran}
\email[Last author]{ghoraishi@gmail.com}

\begin{abstract}
In  this paper we classify  all finite  $2$-groups of
class $2$ for which every automorphism  of order $2$ leaving the Frattini subgroup elementwise fixed is inner.
We prove that every such group $G$ is isomorphic to $Q(n,r)=\langle a,b\mid a^{2^n}=b^{2^r}=1, \,\,\,a^{2^{n-r}}=[a, b]\rangle$ for some positive integers $r,n$ such that $2<2r\leq n$; and every automorphism  of $Q(n,r)$ of  order $2$ leaving the Frattini subgroup elementwise fixed is inner.
\end{abstract}
\maketitle
\section{Introduction}
\label{intro}
Let $p$ be a  prime number and $G$ be a non-abelian  finite  $p$-group.  A
longstanding conjecture \cite{ko}
 asserts  $G$ possesses a non-inner automorphism of order $p$. By the celebrated  result of
W. Gasch\"utz \cite{ga}, non-inner  automorphisms of
$G$  of $p$-power order exist. Deaconescu and Silberberg  \cite{de} proved that a finite $p$-group $G$ satisfying  the condition
$C_G(Z(\Phi(G)))\not=\Phi(G)$ has a noninner automorphism of order $p$ leaving the Frattini subgroup $\Phi(G)$ elementwise fixed.   Liebeck \cite{li}
has shown the same result when
 $G$ is nilpotent of class $2$ and $p>2$. He also showed that the latter  is not
valid for $p=2$  by giving  an example of a finite $2$-group of class $2$ and  order
$128$ in which every automorphism of order $2$
fixing $\Phi(G)$   elementwise is inner.   The first author
\cite{ab2} exhibited  another example of a finite $2$-group of order
$64$ and class $2$ in which every automorphism of order $2$ fixing
$\Phi(G)$  is inner. In \cite{ab1}
it is  shown that  every $2$-group of class $2$ has a non-inner automorphism of order $2$
fixing $\Omega_1(Z(G))$ or $\Phi(G)$ elementwise (actually the same proof in \cite{ab1}
shows that such an automorphism fixes $Z(G)$ or $\Phi(G)$ elementwise).

In this paper we classify all  finite $2$-groups of class $2$ for which every automorphism of order $2$
leaving the Frattini subgroup elementwise fixed is inner.
\begin{theorem}\label{1}
If $G$ is a finite $2$-group of class $2$ for which every automorphism of order $2$ leaving the Frattini subgroup elementwise fixed is inner, then  $G$ is isomorphic to  $\langle a,b\mid a^{2^n}=b^{2^r}=1, \,\,\,a^{2^{n-r}}=[a, b]\rangle, ~2<2r\leq n.$
 \end{theorem}
Throughout the paper $p$ denotes a prime number. For a group $G$,
 $Z(G)$, $G'$, $\Phi(G)$ and  $d(G)$
denote the center, the derived subgroup, the Frattini subgroup and the minimum number of generators of $G$, respectively. The inner automorphism induced by $x\in G$ is denoted by $\theta_x$. If $G$ and $A$ are groups, $Hom(G,A)$ denotes the set of all homomorphisms from $G$ to $A$. If $A$ is an abelian group $Hom(G,A)$ with pointwise multiplication forms  an abelian group.

\section{ Reduction to $2$-generator groups}
Let  $G$ be  a finite $2$-group of class $2$ such
that every automorphism of order $2$  leaving the Frattini
subgroup of $G$ elementwise fixed is inner. Call this condition  $(\star)$.
\begin{remark}\label{Z}(\cite{ab2}, Remark 2.4)
 Let $G$ be a finite $p$-group of class $2$. If $G$ has no non-inner
automorphism of order $p$ leaving $\Phi(G)$ elementwise fixed,
then $Z(G)$ must be cyclic.
\end{remark}
\begin{remark}\label{hom}
If $\Omega_1(Z(G))\leq \Phi(G)$ and $f\in Hom(\frac{G}{\Phi(G)},\Omega_1( Z(G)))$, then it is well-known that the mapping $\phi_f$ defined by $x^{\phi_f}=x\bar{x}^f$ where
$\overline{x}=x\Phi(G)$ for any $x\in G$
, is an automorphism of order $2$ that fixes $\Phi(G)$
elementwise. If  $\phi_f$ is inner, then $\phi_f=\theta_x$ for
some $x\in G$. Since $\theta_x|_{\Phi(G)}=id$ one has $x\in
C_G(\Phi(G))$. If $G$ satisfies the condition $(\star)$, then it follows from \cite{de} that $C_G(Z(\Phi(G)))=\Phi(G)$
 and so $C_G(\Phi(G))=Z(\Phi(G))$. Thus $x\in Z(\Phi(G))$.
Now $(\star)$ implies that
$$d(G)d(Z(G))=d(Hom(\frac{G}{\Phi(G)}, \Omega_1(Z(G)))\leq d(\frac{Z(\Phi(G))}{Z(G)})\leq d(Z(\Phi(G))),$$
therefore we have $d(G)\leq d(Z(\Phi(G)))$ for any group $G$ satisfying the condition $(\star)$.
\end{remark}
\begin{lemma}\label{G=AB}
Let  $G$ be a finite group and  $A\unlhd G$, $B\leq G$ and $G=AB$. Then   $\alpha \in Aut(A)$ and
$\beta\in Aut(B)$ have a common  extension
 to  $G$ if and only if $\alpha$ and $\beta$  agree on  $A\cap B$ and $[a,b]^\alpha= [a^\alpha, b^\beta]$ for all $a\in A, b\in B$.
\end{lemma}
\begin{proof}
It is straightforward.
\end{proof}
\begin{theorem}\label{aut}
Let $G$ be a finite $2$-group of class $2$ and
$\frac{G}{\Phi(G)}=\langle
\overline{x_1}\rangle\times\dots\times\langle\overline{x_d}\rangle$
for some elements $x_1,\dots,x_d\in G$ where
$\overline{x}=x\Phi(G)$ for any $x\in G$.
 If $b_1,\dots ,b_d\in \Omega_1(Z(\Phi(G)))$ such that $[x_i,b_i]=1$ and
$[x_i,b_j]=[x_j,b_i]$ for $1\leq i<j\leq d$,
 then the mapping  $x_i$ to $x_i{b_i}$,  for $1\leq i\leq d$,
 can be extended to an automorphism of $G$ of order $2$.
\end{theorem}
\begin{proof}
For  $1\leq j\leq d$, let $H_j=\langle
x_1,\dots,x_j,\Phi(G)\rangle$. We prove  by induction on $j$ that
$\left|\begin{smallmatrix}x_i\mapsto x_i{b_i}\\ 1\leq i\leq j\end{smallmatrix}\right.$
 determines an automorphisms of order 2
on $H_j$ which fixes the $\Phi(G)$ element-wise. For $j=1$,  if
$h\in H_1$, then $h=x_1^im$, for some non-negative integer $i$ and
$m\in \Phi(G)$. Naturally  define $h^\alpha=(x_1b_1)^im$. Now it
is easy to see  that  $\alpha$ is the desired automorphism on
$H_1$.  Now suppose that $\alpha$ is the automorphism on $H_j$
determined by
$\left|\begin{smallmatrix}x_i\mapsto x_ib_i\\1\leq i\leq j\end{smallmatrix}\right.$
  and $\beta$ is  the automorphism on
$\langle x_{j+1}, \Phi(G)\rangle$ defined by $x_{j+1}\to
x_{j+1}b_{j+1}$.
 Now by Lemma \ref{G=AB}., $\alpha$ and $\beta$ have a common extension to an automorphism on $H_{j+1}$.
\end{proof}
\begin{theorem}
Let $G$ be a finite $2$-group of class $2$ satisfying $(\star)$.
Then $d(G)\leq 3.$
\end{theorem}
\begin{proof}

Let  $\frac{G}{\Phi(G)}=\langle
\overline{x_1}\rangle\times\dots\times\langle\overline{x_{d(G)}}\rangle$
for some elements $x_1,\dots,x_{d(G)}\in G$.
 Define
\begin{align*}
\varphi:\underbrace{\Omega_1(Z(\Phi(G)))\times\cdots\times
\Omega_1(Z(\Phi(G)))}_{d(G)-times} \mapsto
\underbrace{\Omega_1(Z(G))\times\cdots\times
\Omega_1(Z(G))}_{\binom{d(G)+1}{2}-times}
\end{align*}
by $(b_1,\dots,b_{d(G)})^\varphi=
([x_1,b_1],\dots,[x_{d(G)},b_{d(G)}],\dots, \underset{1\leq
i<j\leq d(G)}{[x_i,b_j][b_i,x_j]},\dots)$. It is easy to see that
$\varphi$ is a homomorphism  and if $(b_1,\dots,b_{d(G)})\in
ker(\varphi)$, then by Theorem \ref{aut}. the mapping
$\left|\begin{smallmatrix}x_i\mapsto x_i{b_i}\\1\leq i\leq
j\end{smallmatrix}\right.$ determines an automorphism of order $2$
leaving $\Phi(G)$ elementwise fixed.  Since the  domain of $\varphi$
is elementary abelian, then
$$d(ker(\varphi))\geq d(G)d(Z(\Phi(G)))- \binom{d(G)+1}{2}.$$
By main result of \cite{de} we may assume that $C_G(Z(\Phi(G)))= \Phi(G)$. Therefore it follows from Remark \ref{hom}. that $d(ker(\varphi))\geq \binom{d(G)}{2}$. Now
the condition $(\star)$ implies that $d(G)\geq d(ker(\varphi))\geq
\binom{d(G)}{2}$ and so $d(G)\leq 3$.
\end{proof}
\begin{lemma}
Let $G$ be a finite $2$-group of class $2$ satisfying $(\star)$.
Then $d(G)\neq 3$.
\end{lemma}
\begin{proof}
Suppose that $G= \langle
x_1, x_2, x_3\rangle.$ We may assume  that $G'= \langle
[x_1, x_2]\rangle$, since $Z(G)$ is cyclic by Remark \ref{Z}. Therefore $
[x_1, x_2 ]^i=[x_1, x_3], [x_1, x_2]^j=[x_2, x_3],$ for some integers $i, j.$ Then $[x_1, x_2^{-i}x_3]=1$ and $[x_1^jx_3, x_2]=1
$. Hence $[x_2, x_1^jx_2^{-i}x_3]=1$ and$[x_1, x_1^jx_2^{-i}x_3]=1$ and so $x_1^jx_2^{-i}x_3\in Z(G)\leq \Phi(G)$ (note that by \cite{de}, $Z(G)\leq \Phi(G)$). Therefore $G = \langle x_1, x_2\rangle$, a contradiction. This completes the proof.
\end{proof}
\section{Proof of Theorem $\ref{1}.$}
Let $G$ be a finite $2$-group of class $2$ satisfying condition $(\star)$.
By Section $2$, we may assume that  $d(G)=2$ and $Z(G)$ is cyclic. The proof of Theorem \ref{1}. is as follows:  \\
Y. K. Leong \cite{le} has given a complete classification of
$2$-generator $2$-groups of class $2$ with cyclic center. The classification is as follows:
\begin{itemize}
\item[(1)] $Q(n,r) =\langle a,b\mid
a^{2^n}=b^{2^r}=1, \,\,\,a^{2^{n-r}}=[a, b]\rangle$ and $2r\leq n$.
\item[(2)]$Q(n,r)=\langle a,b\mid
a^{2^n}=b^{2^r}=1, \,\,\,a^{2^r}=[a, b]^{2^{2r-n}},[[a,b],a]=
[[a,b],b]=1\rangle$ and $r\leq n< 2r$.
 \item[(3)]$R(n)=\langle
a,b\mid a^{2^{n+1}}=b^{2^{n+1}}=1, \,\,\,a^{2^n}= [a,
b]^{2^{n-1}}=b^{2^n},[[a,b],a]=[[a,b],b]=1\rangle$ and $ n\geq 1$.
\end{itemize}
Thus, to complete the proof of Theorem \ref{1}. it is enough to check which of these groups satisfy $(\star)$.
To this end,
first note that if a  $2$-generator $p$-group $G$ of class $2$ has a presentation
$$\langle a,b~|~r_\ell(a,b),~1\leq  \ell \leq m \rangle,$$
 then every element  $g\in G$ has the form $g=a^ib^j[a,b]^k$ for some integers $i, j$ and $k$. Moreover, if
$G=\langle a^ib^j[a,b]^k, a^{i'}b^{j'}[a,b]^{k'}\rangle$ and
$r_\ell(a^ib^j[a,b]^k, a^{i'}b^{j'}[a,b]^{k'})=1$, for all $\ell\in\{1,\dots,m\}$, then  by Von Dyck's theorem the mapping $\left|\begin{array}{l}a\mapsto a^ib^j[a,b]^k\\b\mapsto a^{i'}b^{j'}[a,b]^{k'}\end{array}\right.$ determines an automorphism of $G$.
Next, we apply the following result to verify whether a given automorphism of $G$ is inner.
\begin{remark}\label{inn}( \cite{Y}, Part (ii) of Lemma 1)
Suppose that  $G$ is a finite $2$-generator  $2$-group of class $2$, such that $G'=\langle a\rangle$. Then  ~$\alpha \in Aut(G)$ is inner if and only if $[G,\alpha]\leq
G'.$
\end{remark}
We now start to check groups $Q(n,r)$, $R(n)$. \\
\noindent
(1) \; Let $G=Q(n,r) =\langle a,b\mid
a^{2^n}=b^{2^r}=1, \,\,\,a^{2^{n-r}}=[a, b]\rangle$ and $2r\leq n$;
\begin{itemize}
\item[(i)] If $G=Q(2, 1)$, then  $G\simeq D_8$ and does not  satisfy $(\star)$ because it is easy to see that the following map $\alpha$ can be  extended to  a noninner automorphism of order $2$ leaving the Frattini subgroup of $G$ elementwise fixed.
$$\alpha:\left|\begin{matrix}
a\mapsto a^3\\
b\mapsto ab
\end{matrix}
\right.
$$
\item[(ii)] If $G=Q(n, 1)$, then $G$ does not  satisfy $(\star)$ because $G$ has exactly two non-inner automorphisms of order $2$ which act trivially on the Frattini subgroup as follow:
$$\alpha:\left|\begin{array}{l}
a\mapsto a^{1+2^{n-2}+m2^{n-1}}b\\
b\mapsto a^{2^{n-1}}b
\end{array}
\right.,
\; m\in\{0, 1\}
$$
\item[(iii)]  If $n>2$, then every   automorphism of $G$ which is leaving the Frattini subgroup
elementwise fixed, maps the generators $a$ and $b$ as either the map $\alpha_1$ or $\alpha_2$:
$$\alpha_1:\left|\begin{array}{l}
a\mapsto a^{1+m2^{n-1}}\\
b\mapsto a^{s2^{n-1}}b
\end{array}
\right.
,\;
\alpha_2:\left|\begin{array}{l}
a\mapsto a^{1+2^{n-2}+m2^{n-1}}b^{2^{r-1}}\\
b\mapsto a^{s2^{n-1}}b
\end{array}
\right.,
\; m,s\in\{0, 1\}
$$
\end{itemize}
Now Remark \ref{inn}. implies that $\alpha_1 \in Inn(G)$ and $\alpha_2\not\in Inn(G)$. On the other hand $|\alpha_2|\neq 2$. Therefore $G$ satisfies the condition $(\star)$.\\
(2)\;  Let $G=Q(n, r)=\langle a,b\mid
a^{2^n}=b^{2^r}=1, \,\,\,a^{2^r}=[a,
b]^{2^{2r-n}},[[a,b],a]=[[a,b],b]=1\rangle$ where $r\leq n<2r$.
\begin{itemize}
\item [(i)] If $n=r=1$, let
$\alpha:\left|\begin{array}{l}
a\mapsto b\\
b \mapsto a
\end{array}\right.$,
\item[(ii)] If $ n=r>1$ let
$\alpha:\left|\begin{array}{l}
a\mapsto a^{2^{r-1}+1}\\
b\mapsto b^{2^{r-1}+1}
\end{array}\right.$,
\item[(iii)]If $2r>n\geq r+1,\; r>1$, let
$\alpha:\left|\begin{array}{l}
 a\mapsto a^{2^{n-1}-2^{r-1}+1}[a, b]^{2^{2r-n-1}}\\
b\mapsto a^{2^{n-1}}b^{2^{r-1}+1}
\end{array}\right.$,
\end{itemize}
Then it can be verified that in each case $\alpha$ determines a non-inner automorphism of order $2$ leaving
the Frattini subgroup of $G$ elementwise fixed.
\\
 (3)\; Let $G=R(n)=\langle a,b\mid a^{2^{n+1}}=b^{2^{n+1}}=1,
\,\,\,a^{2^n}= [a,
b]^{2^{n-1}}=b^{2^n},[[a,b],a]=[[a,b],b]=1\rangle$, where $n\geq 1$;
 \begin{itemize}
\item [(i)] If $n=1$, let
$\alpha:\left|\begin{array}{l}
a\mapsto ab\\
b\mapsto b^3
\end{array}\right.$,
\item [(ii)]
If $n\geq2$, let
$\alpha:\left|\begin{array}{l}
a\mapsto a^{2^n+2^{n-1}+1}[a, b]^{2^{n-2}}\\
b\mapsto b^{2^n+2^{n-1}+1}[a, b]^{2^{n-2}}\\
\end{array}\right.$.
\end{itemize}
Then it can be verified that in each case $\alpha$ is a non-inner automorphism of order $2$ leaving
the Frattini subgroup of $G$ elementwise fixed.\\
\section*{Acknowledgements}
The authors are grateful to the office of Graduate Studies of the  University of Isfahan for their financial and moral supports.


\begin{thebibliography}{99}
\bibitem {ab1}{A.~ Abdollahi}, \textit {Powerful $p$-groups have non-inner automorphisms of order $p$ and some cohomology,} J. Algebra~\textbf {323}, 779--789, 2010.
\bibitem {ab2}{ Abdollahi, A. } \textit{Finite $p$-groups of class $2$ have noninner automorphisms of order $p$,}  J. Algebra~\textbf {312}, 876--879, 2007.
\bibitem{ad}{ Adney, J. E. and Ti Yen,} \textit{Automorphisms of a $p$-group,}  Illinois J. Math. ~\textbf {9}, 137-143 ,1965.
\bibitem {Y}{Cheng, Y. }\textit{On finite $p$-groups with cyclic commutator subgroup,}   Arch. Math. (Basel) ~\textbf {39}, 295--298, 1982.
\bibitem {de}{Deaconescu, M. and Silberberg, G.~}\textit{ Noninner automorphisms of order $p$ of finite
$p$-groups,}  J.~Algebra~\textbf {250}, 283-287, 2002.
\bibitem {ga}{Gasch\"utz, W.}\textit{ Nichtabelsche $p$-Gruppen besitzen \"aussere $p$-automorphismen,}  J. Algebra~\textbf{ 4}, 1--2, 1966.
\bibitem{le} {Leong, Y. K.}\textit{ Finite 2-groups of class two with cyclic centre,}  J. Austral Math. Soc. Set. (series A) ~\textbf {27}, 125--140, 1978.
\bibitem{li}{Liebeck, H.}\textit{ Outer automorphisms in nilpotent $p$-groups of class $2$,}  J. London Math. Soc.~\textbf{ 40}, 268--275, 1965.
\bibitem {ko}{Mazurov, V. D. and Khukhro$(Eds.)$, E. I. } Unsolved problems in group theory, The Kourovka Notebook ~\textbf {16}, (Russian Academy of Sciences, Siberian Division, Institue of Mathematics, Novosibirisk, 2006).
\end{thebibliography}
\end{document}